\documentclass[12pt]{article}

\usepackage{fullpage}
\usepackage{amsmath}
\usepackage{amsfonts}
\usepackage{verbatim}
\usepackage{amsthm}
\usepackage{amssymb}
\usepackage{epsfig}
\usepackage{scrextend}
\usepackage{float}
\usepackage{scalerel}
\usepackage{graphicx}

\newtheorem{theorem}{Theorem}[section]
\newtheorem{proposition}[theorem]{Proposition}

\newtheorem{example}[theorem]{Example}

\newcommand{\inv}[0]{^{-1}}
\newcommand{\mc}[0]{\mathcal}
\newcommand{\dual}[0]{^*}
\newcommand{\bb}[0]{\mathbb}
\newcommand{\bs}[0]{\backslash}
\newcommand{\vp}[0]{\varphi}
\newcommand{\exterior}[0]{{\rm Exterior}}
\newcommand{\interior}[0]{{\rm Interior}}
\newcommand{\lattice}[0]{{\rm Lattice}}

\makeatletter
\def\blfootnote{\xdef\@thefnmark{}\@footnotetext}
\makeatother

\title{Cobiased graphs: Single-element extensions and elementary quotients of graphic matroids}

\author{{Daniel Slilaty}
       \thanks{E-mail address: {\tt daniel.slilaty@wright.edu}}\\
   {\small Department of Mathematics and Statistics} \\
   {\small Wright State University}\\
   {\small Dayton, OH, USA}\\
\and {Thomas Zaslavsky}
    \thanks{E-mail address: {\tt zaslav@math.binghamton.edu}}\\
   {\small Department of Mathematics and Statistics} \\
   {\small Binghamton University}\\
   {\small Binghamton, NY, USA}\\}

\begin{document}

\maketitle

{}\blfootnote{Mathematic Subject Classification 2020: Primary 05B35; Secondary 05C22.}

\begin{abstract}
Zaslavsky (1991) introduced a graphical structure called a biased graph and used it to characterize all single-element coextensions and elementary lifts of graphic matroids. We introduce a new, dual graphical structure that we call a cobiased graph and use it to characterize single-element extensions and elementary quotients of graphic matroids.
\end{abstract}

\tableofcontents

%%%%%%%%%%%%%%%%%%%%%%%%%%%%%%%%%%%%%%%%%%%%%%%%%%%%%%%%

\section{Introduction}

An \emph{elementary lift} of a matroid $M$ is a matroid of the form $N\bs e_0$ in which $N/e_0=M$ and $e_0$ is neither a loop nor coloop of $N$. Single-element coextensions and elementary lifts of graphic matroids were characterized in terms of graphic structures by Zaslavsky \cite{Zaslavsky:BiasedGraphs1, Zaslavsky:BiasedGraphs2}. Aside from the work in \cite{Zaslavsky:BiasedGraphs1, Zaslavsky:BiasedGraphs2}, single-element coextensions and elementary lifts of graphic matroids have been objects of consistent interest in matroid theory and related fields. Guenin's investigation \cite{Guenin:IntegralPolyhedra} into integral polyhedra related to binary elementary lifts of graphic matroids is notable.

An \emph{elementary quotient} of a matroid $M$ is a matroid of the form $N/e_0$ in which $N\bs e_0=M$ and $e_0$ is neither a loop nor coloop of $N$. Elementary quotients of graphic matroids have also been of consistent interest, in particular binary elementary quotients. Guenin's result in \cite{Guenin:IntegralPolyhedra} applies not only to binary elementary lifts of graphic matroids but also binary elementary quotients. Guenin, Pivotto, and Wollan \cite{GueninPivottoWollan:Relationships} explored the relationships between binary elementary lifts and quotients of graphic matroids. Seymour's original proof of the decomposition theorem for regular matroids \cite{Seymour:Decompositions} uses binary single-element extensions and elementary quotients of graphic matroids. In the field of error-correcting codes, Hakimi and Bredeson \cite{HakimiBredeson} constructed binary codes using circuit spaces of binary single-element and multiple-element extensions and quotients of graphic matroids. Jungnickel and Vanstone \cite{JungnickelVanstone:QAry} do the same with $q$-ary codes.

Despite all of the interest in single-element extensions and elementary quotients of graphic matroids, there has been no general description of them. Recski characterized connected single-element extensions and elementary quotients of graphic matroids \cite{Recski1,Recski2} that are representable over a given field and made some generalizations. In this paper we will fully characterize all single-element extensions and elementary quotients of graphic matroids in terms of graphical structures.

\section{Cobiased Graphs}

Let $G$ be a graph and $X\subset V(G)$ or $X\subset G$. The \emph{coboundary} of $X$ is denoted by $\delta(X)$ and is the set of links (i.e., non-loop edges) of $G$ with exactly one endpoint in $X$. Consider a tripartition $\{X_1,X_2,X_3\}$ of the vertices of one connected component of $G$ into nonempty subsets such that each induced subgraph $G[X_i]$ is connected and there is at least one edge of $G$ connecting each pair of these three subgraphs. The union of the three bonds (i.e., minimal edge cuts) $B_1=\delta(X_1)$, $B_2=\delta(X_2)$, and $B_3=\delta(X_3)$ is called a \emph{tribond} (see Figure \ref{F:TriangularTriple}).

\begin{figure}[H]
\begin{center}
\includegraphics[page=1,scale=0.7]{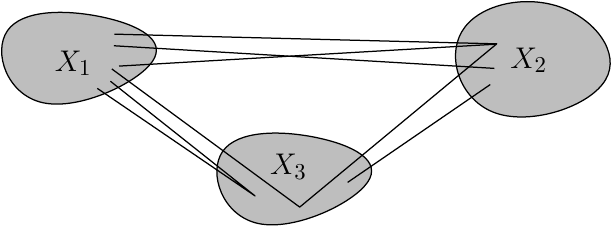}
\end{center}
\caption{A tribond.}\label{F:TriangularTriple}
\end{figure}

A \emph{linear class} of bonds of $G$ is a subset $\mc L$ of the set of all bonds in $G$ satisfying the property that every tribond contains zero, one, or three bonds from $\mc L$; that is, a tribond cannot contain exactly two bonds from $\mc L$. We call the pair $(G,\mc L)$ in which $G$ is a graph and $\mc L$ a linear class of bonds a \emph{cobiased graph}. Bonds in $\mc L$ are called \emph{cobalanced} and bonds not in $\mc L$ are called \emph{un-cobalanced}.
(The name comes from the fact that cobalance of bonds is dual to balance of cycles in \cite{Zaslavsky:BiasedGraphs1} \emph{et seq}.)
A linear class of bonds $\mc L$ is \emph{trivial} when all bonds are cobalanced; that is, $\mc L$ is the set of all bonds of $G$.

\section{Join and Complete Join Matroids of Cobiased Graphs}

In this section we will describe two matroids associated with a cobiased graph $(G,\mc L)$. These matroids are denoted by $J_0(G,\mc L)$ and $J(G,\mc L)$ and are called respectively the \emph{complete join} and \emph{join} matroids of $(G,\mc L)$. The term ``join" is used to echo the considerable amount of literature on $T$-joins of graphs that are dependent sets of binary elementary quotients of graphic matroids. The matroids $J_0(G,\mc L)$ and $J(G,\mc L)$ are defined in Section \ref{S:Cocircuits} in terms of their cocircuits using Crapo's Theorem \cite[p.\ 62]{Crapo:Extensions} on single-element extensions of matroids. Crapo's Theorem also immediately implies that $J_0(G,\mc L)$ and $J(G,\mc L)$ characterize respectively single-element extensions and elementary quotients of graphic matroids. From there we will determine the bases, independent sets, rank functions, and circuits of these two matroids. We will also provide a graphical description of deletions and contractions.

\subsection{Cocircuits} \label{S:Cocircuits}

A pair of bonds $B_1,B_2$ in a graph $G$ is a \emph{modular pair} when the number of connected components of $G-(B_1\cup B_2)$ is two more than the number of connected components of $G$. Thus $B_1,B_2$ form a modular pair of bonds when: $B_1\cup B_2$ forms a tribond, $B_1\cup B_2$ form the configuration in Figure \ref{F:TriangularTriple} but with no edges between $X_1$ and $X_2$ (see Figure \ref{F:ModularBonds}, left), or $B_1$ and $B_2$ are in two distinct connected components of $G$ (see Figure \ref{F:ModularBonds}, right). When $B_1,B_2$ is a modular pair of bonds but $B_1\cup B_2$ is not a tribond (i.e., $B_1\cup B_2$ is one of the structures from Figure \ref{F:ModularBonds}) we will call $B_1\cup B_2$ a \emph{dibond}. Note that a modular pair of bonds share an edge if and only if their union is a tribond.

\begin{figure}[H]
\begin{center}
\includegraphics[page=2,scale=0.4]{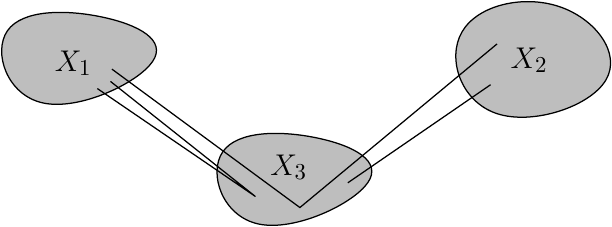}\hspace{2cm}
\includegraphics[page=3,scale=0.8]{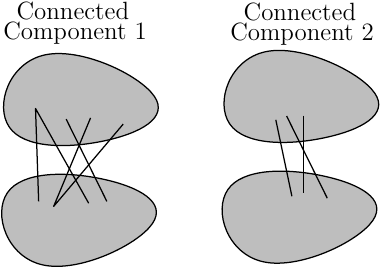}
\end{center}
\caption{Every dibond is of exactly one of two possible types.}\label{F:ModularBonds}
\end{figure}

The matroid of Theorem \ref{T:Cocircuits} is $J_0(G,\mc L)$, the complete join matroid of $(G,\mc L)$. If $\mc L$ is trivial, then the single-element extension associated with $\mc L$ is just $M(G)$ along with a new element that is either a loop or coloop.

\begin{theorem} \label{T:Cocircuits}
If $\mc L$ is a non-trivial linear class of bonds of $G$, then there is a matroid with element set $E(G)\cup e_0$ in which $e_0$ is neither a loop nor a coloop and whose cocircuits consist of the following:
\begin{itemize}
\item[(1)] bonds in $\mc L$,
\item[(2)] sets of the form $B\cup e_0$ in which $B$ is a bond not in $\mc L$, and
\item[(3)] tribonds and dibonds which do not contain a bond from $\mc L$.
\end{itemize}
Conversely, if $N$ is a single-element extension of the graphic matroid $M(G)$ with new element $e_0$ which is neither a loop nor a coloop of $N$, then there is a non-trivial linear class of bonds $\mc L$ of $G$ such that the cocircuits of $N$ consist of the sets above.
\end{theorem}
\begin{proof}
This follows the cocircuit version of Crapo's Theorem \cite[p.62]{Crapo:Extensions} on single-element extensions of matroids as long as modularity of bonds as we have defined them graphically is exactly how modular pairs of cocircuits behave in the graphic matroid $M(G)$. This will complete our proof.

Let $B_1$ and $B_2$ be distinct bonds in $G$ and let $H_1=E-B_1$ and $H_2=E-B_2$. Now $B_1,B_2$ is a modular pair of cocircuits $M(G)$ if and only if $H_1,H_2$ is a modular pair of hyperplanes in $M(G)$, if and only if $r(H_1)+r(H_2)=r(H_1\cup H_2)+r(H_1\cap H_2)$, if and only if $r(H_1\cap H_2)=r(M(G))-2$, if and only if the flat $H_1\cap H_2$ has 2 more connected components than does $G$, which is how we defined a modular pair of bonds.
\end{proof}

The \emph{join matroid} of a cobiased graph $(G,\mc L)$ is defined as $J(G,\mc L)=J_0(G,\mc L)/e_0$.  If $\mc L$ is trivial, then $J(G,\mc L)=M(G)$. Theorem \ref{T:Cocircuits2} is an immediate corollary of Theorem \ref{T:Cocircuits}.

\begin{theorem} \label{T:Cocircuits2}
If $\mc L$ is a non-trivial linear class of bonds in $G$, then the cocircuits of $J(G,\mc L)$ consist of the following:
\begin{itemize}
\item[(1)] bonds in $\mc L$ and
\item[(2)] tribonds and dibonds which do not contain a bond from $\mc L$.
\end{itemize}
Furthermore, if $N$ is an elementary quotient of a graphic matroid $M(G)$, then $N=J(G,\mc L)$ for some non-trivial linear class $\mc L$.
\end{theorem}

\subsection{Bases and independent sets} \label{S:Bases}

Consider a partition $\pi=\{X_1,\ldots,X_k\}$ of $V(G)$ into nonempty parts such that each induced subgraph $G[X_i]$ is connected. Denote the set of such partitions for $G$ by $\lattice(G)$. The partial ordering is the usual refinement partial ordering; that is, given two such partitions $\pi_1=\{X_1,\ldots,X_k\}$ and $\pi_2=\{Y_1,\ldots,Y_l\}$ in $\lattice(G)$, we have $\pi_2\leq\pi_1$ when for each part $Y_i\in\pi_2$ there is a part $X_j\in\pi_1$ such that $Y_i\subseteq X_j$. It is well known that with this partial ordering $\lattice(G)$ is indeed a lattice. 
If $H$ is a subgraph of $G$ or subset of $E(G)$, then $H$ naturally induces a partition $\pi_H\in\lattice(G)$ corresponding to the connected components of $H\cup V(G)$.
If $H\subseteq G$ with $V(H)=V(G)$, then $\lattice(H)$ is a join subsemilattice of $\lattice(G)$. (For the proof, let $\pi \in \lattice(G)$.  Think of the edges of $\cup G[X_i]$ as a relation on $V(G)$ and extend it to an equivalence relation $\equiv_\pi$.  If $\tau \in \lattice(G)$, extend $\equiv_\pi \cup \equiv_\tau$ to an equivalence relation $\equiv$; then the join $\pi \vee \tau$ is the partition that corresponds to $\equiv$.  Supposing that $\pi,\tau \in \lattice(H)$, this formula for the join is the same whether viewed in $H$ or in $G$.) 
In particular $\lattice(G)$ is a join subsemilattice of $\lattice(K_n)$, which is the usual partition lattice of the set $V(K_n)$.

Given $\pi=\{X_1,\ldots,X_k\}\in\lattice(G)$, denote the set of edges in $G[X_1]\cup\dots\cup G[X_k]$ by $\interior(\pi)$, or $\interior_G(\pi)$ when necessary. The collection of such edge sets is of course exactly the set of flats of the graphic matroid $M(G)$. The set of edges of $G$ which are not in $\interior(\pi)$ is denoted by $\exterior(\pi)$ or $\exterior_G(\pi)$. Note that $\exterior(\pi)$ is a union of bonds.

Given a cobiased graph $(G,\mc L)$, we call $\pi\in\lattice(G)$ \emph{cobalanced} with respect to $\mc L$ when every bond in $\exterior(\pi)$ is cobalanced; otherwise the partition is \emph{un-cobalanced}. The maximal element of $\lattice(G)$ is $\pi_G$, that is, the partition of $V(G)$ given by the connected components of $G$ itself. The coatoms of $\lattice(G)$, that is, the elements of $\lattice(G)$ which are covered by $\pi_G$, are those partitions $\pi$ for which $\exterior(\pi)$ is a bond.

\begin{theorem} \label{T:Bases}
If $\mc L$ is a non-trivial linear class of bonds of $G$, then the bases of $J_0(G,\mc L)$ consist of the following:
\begin{itemize}
\item[(1)] edge sets of maximal forests of $G$ and
\item[(2)] sets of the form $F\cup e_0$ in which $F$ is a maximal forest with one edge deleted such that the bond $\exterior(\pi_F)$ is un-cobalanced.
\end{itemize}
\end{theorem}

When one edge is deleted from a maximal forest, one component tree is broken into two trees. The bond between those two trees is $\exterior(\pi_F)$.

\begin{proof}
For a general matroid $M$, $B$ is a basis if and only if $B$ is a minimal set which intersects every cocircuit (see, for example, \cite[p.77]{Oxley:2ndEdition}). Now since edge sets of cycles are dependent in $M(G)$, they are also dependent in its single-element extension $J_0(G,\mc L)$. So if $B$ is a base of $J_0(G,\mc L)$, then $B\bs e_0$ is the edge set of a forest in $G$. Theorem \ref{T:Cocircuits} now implies the following: $e_0\notin B$ if and only if $B$ is a maximal forest and $e_0\in B$ if and only if $B\bs e_0$ is obtained from a maximal forest by the deletion of one edge $e$ such that the bond $\exterior(\pi_F)$ is un-cobalanced.
\end{proof}

\begin{theorem} \label{T:Bases2}
If $\mc L$ is a non-trivial linear class of bonds of $G$, then the bases of $J(G,\mc L)$ consist of the
forests $F$ for which $\exterior(\pi_F)$ is an un-cobalanced bond.
\end{theorem}
\begin{proof}
Because $\mc L$ is non-trivial, $e_0$ is not a loop or coloop of $J_0(G,\mc L)$. Thus the bases of $J(G,\mc L)$ are obtained from the bases of $J_0(G,\mc L)$ that contain $e_0$ by removing $e_0$. The result now follows from Theorem \ref{T:Bases}.
\end{proof}

\begin{theorem} \label{T:Independence}
If $\mc L$ is a non-trivial linear class of bonds of $G$, then the independent sets of $J_0(G,\mc L)$ consist of the following:\begin{itemize}
\item[(1)] Edge sets of forests.
\item[(2)] Edge sets of the form $F\cup e_0$ in which $F$ is a forest and $\pi_F$ is un-cobalanced.
\end{itemize}
\end{theorem}
\begin{proof}
This follows from Theorem \ref{T:Bases} because (1) and (2) describe exactly the subsets of the bases of $J_0(G,\mc L)$.
\end{proof}

\begin{theorem} \label{T:Independence2}
If $\mc L$ is a non-trivial linear class of bonds of $G$, then the independent sets of $J(G,\mc L)$ consist of the edge sets of forests $F$ such that $\pi_F$ is un-cobalanced.
\end{theorem}
\begin{proof}
This follows from Theorem \ref{T:Bases2} because these are exactly the subsets of the bases of $J(G,\mc L)$.
\end{proof}

\subsection{Rank} \label{S:Rank}

If $H$ is a subgraph of $G$ or subset of $E(G)$, then $|\pi_H|$ is the number of connected components of $H\cup V(G)$. Thus if $X\subseteq E(G)$, then $r_{M(G)}(X)=|V(G)|-|\pi_X|$.

\begin{theorem} \label{T:Rank}
If $\mc L$ is a non-trivial linear class of bonds of $G$ and $X\subseteq E(G)$, then
\begin{itemize}
\item[(1)] $r_{J_0(G,\mc L)}(X)=|V(G)|-|\pi_X|$,
\item[(2)] $r_{J_0(G,\mc L)}(X\cup e_0)=|V(G)|-|\pi_X|$ when $\pi_X$ is cobalanced, and
\item[(3)] $r_{J_0(G,\mc L)}(X\cup e_0)=|V(G)|-|\pi_X|+1$ when $\pi_X$ is un-cobalanced.
\end{itemize}
\end{theorem}
\begin{proof}
Part (1) follows from the fact that the rank of $X$ in $M(G)$ and its single-element extension $J_0(G,\mc L)$ must be the same. Theorem \ref{T:Independence} implies that there is a circuit containing $e_0$ in the set $X\cup e_0$ if and only if $\pi_X$ is cobalanced. This implies Parts (2) and (3).
\end{proof}

\begin{theorem} \label{T:Rank2}
If $\mc L$ is a non-trivial linear class of bonds of $G$ and $X\subseteq E(G)$, then
\begin{itemize}
\item[(1)] $r_{J(G,\mc L)}(X)=|V(G)|-|\pi_X|-1$ when $\pi_X$ is cobalanced,
\item[(2)] $r_{J(G,\mc L)}(X)=|V(G)|-|\pi_X|$ when $\pi_X$ is un-cobalanced.
\end{itemize}
\end{theorem}
\begin{proof}
This follows from Theorem \ref{T:Rank} and the fact that $J(G,\mc L)=J_0(G,\mc L)/e_0$.
\end{proof}

\subsection{Circuits} \label{S:Circuits}

If $J$ is a forest in $(G,\mc L)$ for which $\pi_J$ is cobalanced and $J$ is minimal with respect to this property, then $J$ is called an $\mc L$-\emph{join} of the cobiased graph $(G,\mc L)$.  Being an $\mc L$-join means that deleting any edge of $J$ creates a bond that is not in $\mc L$.

\begin{theorem} \label{T:Circuits}
If $\mc L$ is a non-trivial linear class of bonds of $G$, then the circuits of $J_0(G,\mc L)$ consist of the following:
\begin{itemize}
\item[(1)] edges sets of the form $J\cup e_0$ in which $J$ is an $\mc L$-join and
\item[(2)] edges sets of cycles.
\end{itemize}
\end{theorem}
\begin{proof}
Edge sets of cycles are circuits of $M(G)$ and therefore are also circuits in $J_0(G,\mc L)$. Hence, any other circuit consists of the edge set of some forest along with $e_0$. Suppose that $F$ is the edge set of a forest for which $F\cup e_0$ is a circuit. Theorem \ref{T:Independence} implies that $F\cup e_0$ is independent when $\pi_F$ is un-cobalanced and dependent when $\pi_F$ is cobalanced. Thus $F\cup e_0$ is a circuit when $\pi_F$ is cobalanced and $F$ is minimal with respect to this property. Thus $F$ is an $\mc L$-join. \end{proof}

\begin{theorem} \label{T:Circuits2}
If $\mc L$ is a non-trivial linear class of bonds of $G$, then the circuits of $J(G,\mc L)$ consist of the following:
\begin{itemize}
\item[(1)] edge sets of $\mc L$-joins and
\item[(2)] edge sets of cycles that do not contain $\mc L$-joins.
\end{itemize}
\end{theorem}
\begin{proof}
This follows from Theorem \ref{T:Circuits} and the fact that $J(G,\mc L)=J_0(G,\mc L)/e_0$.
\end{proof}

\subsection{Deletions and contractions} \label{S:DeletionsContractions1}

Let $G$ be a graph and $e$ a link in $G$. The bonds of $G/e$ are the bonds of $G$ which do not contain $e$. Define $(G,\mc L)/e=(G/e,\mc L/e)$ in which $\mc L/e$ is the set of all bonds $B\in\mc L$ which do not contain $e$. Now $(G/e,\mc L/e)$ is a cobiased graph because any tribond of $G/e$ is a tribond of $G$.

The situation for deletions is only slightly more complicated. If $B$ is a bond in $G\bs e$, then either $B$ or $B\cup e$ is a bond in $G$. We define $(G,\mc L)\bs e=(G\bs e,\mc L\bs e)$ in which $\mc L\bs e$ is the set of all bonds $B$ in $G\bs e$ for which either $B$ or $B\cup e\in\mc L$. Now $(G\bs e,\mc L\bs e)$ is a cobiased graph because if $T$ is a tribond of $G\bs e$, then either $T$ or $T\cup e$ is a tribond of $G$.

\begin{theorem} \label{T:Minors}
If $(G,\mc L)$ is a cobiased graph and $e$ is a link in $G$, then \begin{itemize}
\item[(1)] $[J_0(G,\mc L)]\bs e=J_0(G\bs e,\mc L\bs e)$,
\item[(2)] $[J(G,\mc L)]\bs e=J(G\bs e,\mc L\bs e)$,
\item[(3)] $[J_0(G,\mc L)]/e=J_0(G/e,\mc L/e)$, and
\item[(4)] $[J(G,\mc L)]/e=J(G/e,\mc L/e)$.
\end{itemize}
\end{theorem}
\begin{proof}
(2) We compare dependent sets. Let $D\subseteq E(G)$. If $D$ is a dependent set of $[J(G,\mc L)]\bs e$, then $e\notin D$ and either $D$ contains a cycle or $D$ is a forest such that $\pi_D\in\lattice(G)$ is a cobalanced partition. If $D$ contains a cycle, then it is a dependent set in $J(G\bs e,\mc L\bs e)$. If $D$ is a forest, then because $e\notin D$, the partition of $V(G)$ associated with $D$ is the same for $G$ and $G\bs e$. Thus $\pi_D$ is still a cobalanced partition in $\lattice(G\bs e)$, so $D$ is a dependent set of $J(G\bs e,\mc L\bs e)$. Conversely, if $D$ is a dependent set of $J(G\bs e,\mc L\bs e)$, then either $D$ contains a cycle, in which case it is dependent in $[J(G,\mc L)]\bs e$, or $D$ is a forest of $G\bs e$ such that $\pi_D\in\lattice(G\bs e)$ is cobalanced. Again, the partition of $V(G)$ associated with $D$ is the same in $G$ as in $G\bs e$, so $\pi_D\in\lattice(G)$ is cobalanced, from which it follows that $D$ is dependent in $[J(G,\mc L)]\bs e$.

\smallskip\noindent(1)  The proof is similar to that of Part (2) with only the added detail of noting the presence of $e_0$ in dependent sets without cycles.

\smallskip\noindent(3 and 4) These follow a similar strategy to the proofs of (1) and (2) but by comparing cocircuits. The details are left to the reader.
\end{proof}

\subsection{Vertex union} \label{S:CutVertices}

An operation that preserves graphic matroids is the union of two graphs at a single vertex.  This operation has the same property for cobiased graphs and join matroids, as we see in Theorem \ref{T:CutVertices}.  That fact is hinted at within discussions in \cite{Recski2} but is not fully developed.

\begin{theorem} \label{T:CutVertices}
If $G_1=H\cup K$ in which $H\cap K$ is a single vertex and $G_2=H\cup K$ in which $H\cap K$ is empty, then
\begin{itemize}
\item[(1)] $\mc L$ is a linear class of bonds in $G_1$ if and only if $\mc L$ is a linear class of bonds in $G_2$ and
\item[(2)] $J_0(G_1,\mc L)=J_0(G_2,\mc L)$ and $J(G_1,\mc L)=J(G_2,\mc L)$.
\end{itemize}
\end{theorem}
\begin{proof}
The set of bonds in $G_1$ is exactly the set of bonds in $G_2$; furthermore, the tribonds of $G_1$ are the same as the tribonds of $G_2$ because no tribond can have edges in more than one block of a graph. This proves (1). Part (2) follows from these same facts and comparing cocircuits.
\end{proof}

\subsection{Two simple examples}\label{S:TwoPointCobias}

\begin{example}
Pick vertices $a, b$ in $G$.  Define a bond as cobalanced when it does not separate $a$ and $b$.  Let $\mc L_{a,b}$ denote this set of cobalanced bonds.

\begin{proposition}  $(G,\mc L_{a,b})$ is a cobiased graph.
\end{proposition}

\begin{proof} Consider a tribond corresponding to a tripartition $\{X,Y,Z\}$ of $V(G)$.  If $a, b$
are in the same part, all three bonds are cobalanced.  If $a\in X$ and $b\in Y$, then the bond $\delta (Z)$ is cobalanced and the other two bonds are not.  That is, an even number of the three are un-cobalanced.  It follows that $\mc L_{a,b}$ is a linear class of bonds.
\end{proof}

The linear class $\mc L_{a,b}$ has a special property:  In every tribond the number of un-cobalanced bonds is even.   This looks like a dual of the characteristic property of biased graphs derived from edge signs (gains in the 2-element group), that in every theta graph the number of unbalanced cycles is even (called ``additive bias'' in \cite{Zaslavsky:BiasedGraphs1}).  By analogy, let us call such a linear class of bonds \emph{additive}.  Signed graphs are especially simple (and important) gain graphs.  That raises the following two questions:  Is there similar importance for additive cobias?  Is there a simple general construction of additively cobiased graphs, dual in some sense to the construction of additively biased graphs from signed graphs?  Example \ref{X:Nonseparating} shows that $\mathbb{Z}_2$ does in fact relate to the class $\mc L_{a,b}$, which is a step in the direction of answers.
\end{example}

\begin{example}\label{X:NonseparatingW}
More generally let $W \subseteq V(G)$ have even cardinality, define a bond to be \emph{cobalanced} if it separates $W$ into two even subsets, and let ${\mc L}_W^+$ be the linear class of all such bonds.  Then ${\mc L}_W^+$ is also additive.
\end{example}

Both examples are forms of quotient labeled cobias; see Example \ref{X:Nonseparating}.

\section{Gain graphs and linear classes of bonds}

An \emph{oriented edge} in a graph is an edge $e$ along with a chosen direction along that edge. If $e$ is an oriented edge, then the reverse orientation is denoted by $-e$ when using additive notation and by $e\inv$ when using multiplicative notation. (We will be using additive notation except in Sections \ref{S:Planar} and \ref{S:PlanarNotRealizable}.) The set of all possible oriented edges in $G$ is denoted by $\vec E(G)$.
An \emph{oriented bond} $\vec B$ is a bond $B=\delta(X)$ along with a choice of orientation for the edges of $B$, either all away from $X$ or all towards $X$. The reverse orientation of $\vec B$ is denoted by $-\vec B$ when using additive notation and $\vec B^{-1}$ when using multiplicative notation.

Let $\Gamma$ be a group. A $\Gamma$-\emph{gain graph} is a pair $(G,\vp)$ in which $G$ is a graph and $\vp\colon E(G)\to\Gamma$ is a mapping such that $\vp(-e)=-\vp(e)$ for an additive group and $\vp(e\inv)=\vp(e)\inv$ for a multiplicative group. (All our additive groups are abelian.  Our multiplicative groups are not assumed to be abelian.) The function $\vp$ is called a $\Gamma$-\emph{gain function}.  Gain graphs give rise to biased graphs (see \cite{Zaslavsky:BiasedGraphs1}).  Similarly, they give rise to cobiased graphs, though not always.

\subsection{Cobiased graphs from additive gain graphs}

If $\Gamma$ is an additive group and $(G,\vp)$ is a $\Gamma$-gain graph, then for each oriented bond $\vec B$, define $\vp(\vec B)=\sum_{e\in\vec B}\vp(e)$. Say that the bond $B$ is \emph{cobalanced} when $\vp(\vec B)=-\vp(-\vec B)=0$. Let $\mc L_\vp$ be the set of cobalanced bonds of $(G,\vp)$.

\begin{proposition} \label{P:Gains1}
If $\Gamma$ is an additive group and $(G,\vp)$ is a $\Gamma$-gain graph, then $(G,\mc L_\vp)$ is a cobiased graph.
\end{proposition}
\begin{proof}
Consider a tribond containing the three bonds $\delta(X_1)$, $\delta(X_2)$, $\delta(X_3)$ and assume without loss of generality that $\delta(X_1),\delta(X_2)\in\mc L_\vp$. Let $\vec B_1$ and $\vec B_2$ be the oriented bonds obtained from $\delta(X_1)$ and $\delta(X_2)$ by orienting all edges away from $X_1$ and $X_2$. Let $\vec B_3$ be the oriented bond obtained from $\delta(X_3)$ by orienting all of its edges towards $X_3$. Now $\vp(\vec B_3)=\vp(\vec B_1)+\vp(\vec B_2)=0$, which implies that a tribond cannot have exactly two bonds in $\mc L_\vp$, which is our result.
\end{proof}

\subsection{Cobiased planar graphs using gains over arbitrary groups} \label{S:Planar}

Let $\Gamma$ be a multiplicative group (not necessarily abelian), let $G$ be a connected graph embedded in the plane, and let $(G,\vp)$ be a $\Gamma$-gain graph. Consider an oriented bond $\vec B$ in $G$. The oriented edges of $\vec B$ correspond to a closed walk in the topological dual graph $G\dual$. Thus there is a well-defined cyclic ordering of the edges of $\vec B$ up to a choice of a starting edge and clockwise or counterclockwise direction. So, given such an oriented bond $\vec B$, let $e_1,\ldots,e_k$ be a cyclic ordering with $e_1$ as the starting edge. Define $\vp(\vec B)=\vp(e_1)\cdots\vp(e_k)$. Note that any other choice of starting edge yields a product $\vp(e_i)\cdots\vp(e_k)\vp(e_1)\cdots\vp(e_{i-1})$, which is conjugate to $\vp(e_1)\cdots\vp(e_k)$ in the group $\Gamma$. Furthermore, a different choice of direction yields a product that is the inverse of the original. Therefore, $\vp(\vec B)=1$ for any one choice of starting edge and direction if and only if $\vp(\vec B)=1$ for all possible choices of starting edge and direction. Say that a bond $B$ is \emph{cobalanced} when $\vp(\vec B)=1$ for some choice of starting edge and direction and let $\mc L_\vp$ be the set of cobalanced bonds given by $\vp$.

\begin{proposition} \label{P:Gains2}
If $G$ is a graph embedded in the plane, $\Gamma$ is a multiplicative group, and $(G,\vp)$ is a $\Gamma$-gain graph, then $(G,\mc L_\vp)$ is a cobiased graph.
\end{proposition}
\begin{proof}
This proof is similar to the one for Proposition \ref{P:Gains1}, but with the added concern of picking starting edges and directions for each bond in a tribond to match with the others.
\end{proof}

\subsection{An example not realizable by gains}  \label{S:PlanarNotRealizable}

The example here is essentially the topological dual of \cite[Example 5.8]{Zaslavsky:BiasedGraphs1}. Consider the labeled graph $G\cong K_{2,4}$ shown on the left in Figure \ref{F:NonGroupExample} with all edges oriented in the downward direction. Let $\mc L=\{a_1a_2a_3a_4,a_1a_2b_3b_4,b_1b_2a_3a_4\}$. Up to reembedding of $G$, a tribond in $G$ is of one of the two types shown on the right of Figure \ref{F:NonGroupExample}.  Note that any such tribond contains at most one bond from $\mc L$. Thus $(G,\mc L)$ is a cobiased graph.

\begin{figure}[H]
\begin{center}
\includegraphics[page=4,scale=.8]{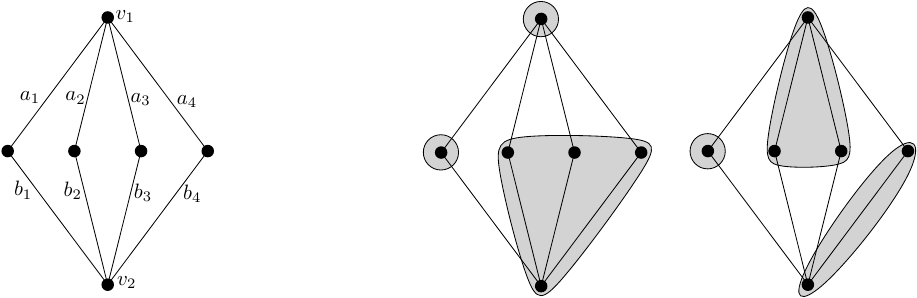}
\end{center}
\caption{All edges are oriented in the downward direction. Every tribond is of one of the two types shown.}\label{F:NonGroupExample}
\end{figure}

By way of contradiction, assume that there is a multiplicative gain function $\vp$ for which $\mc L=\mc L_\vp$. For simplicity, let us denote $\vp(a_i)$ and $\vp(b_i)$ by $a_i$ and $b_i$. Therefore $a_1a_2a_3a_4=1$, $a_1a_2b_3b_4=1$, and $b_1b_2a_3a_4=1$. Thus
\[1=(a_1a_2b_3b_4)\inv a_1a_2a_3a_4(b_1b_2a_3a_4)\inv=(b_1b_2b_3b_4)\inv,\]
which implies that $b_1b_2b_3b_4\in\mc L$, a contradiction.

\subsection{Cycle shifting}

Let $(G,\vp)$ be a $\Gamma$-gain graph in which $\Gamma$ is an additive group. Let $C$ be a cycle in $G$ with oriented edges $e_1,\ldots,e_k$ in cyclic order in one direction around $C$. For any $a\in\Gamma$, let $\psi_{{C,a}}$ be the $\Gamma$-gain function on $G$ for which $\vp(e_i)=a$ and $\vp(e)=0$ for all $e$ such that $\pm e\notin\{e_1,\ldots,e_k\}$. Now for any oriented bond $\vec B$, $\vp(\vec B)=(\vp+\psi_{{C,a}})(\vec B)$. Thus $\mc L_\vp=\mc L_{\vp+\psi_{{C,a}}}$. We call this operation \emph{shifting}. We say that two $\Gamma$-gain functions $\vp_1$ and $\vp_2$ (or two $\Gamma$-gain graphs $(G,\vp_1)$ and $(G,\vp_2)$) are \emph{shifting equivalent} when $\vp_1$ is obtained from $\vp_2$ via a sequence of shifts. Note that this relation is an equivalence relation.

\begin{theorem} \label{T:ShiftingEquivalent}
If $\Gamma$ is an additive group and $(G,\vp)$ and $(G,\psi)$ are $\Gamma$-gain graphs, then $(G,\vp)$ and $(G,\psi)$ are shifting equivalent if and only if $\vp(\vec B)=\psi(\vec B)$ for all oriented bonds $\vec B$ in $G$. \end{theorem}

In order to prove Theorem \ref{T:ShiftingEquivalent}, we need a concept of normalizing gains. So if $(G,\vp)$ is a $\Gamma$-gain graph and $T$ is a maximal forest in $G$, then we say that $\vp$ is $T$-\emph{normalized} when $\vp$ is zero on each edge of $G$ not in $T$.

\begin{proposition} \label{P:Normalizing}
Let $\Gamma$ be an additive group, $(G,\vp)$ a $\Gamma$-gain graph, and $T$ a maximal forest of\/ $G$.  Then there is a unique $T$-normalized $\Gamma$-gain function $\vp_T$ that is shifting equivalent to $\vp$.
\end{proposition}

\begin{proof}
If $e$ is an edge outside $T$, let $C(e)$ be the fundamental cycle in $T\cup e$. Now $\vp_T = \vp-\sum_{e\notin T}\psi_{{C(e),\vp(e)}}$ is a $\Gamma$-gain function that is shifting equivalent to $\vp$ and is zero outside $T$.

We prove that $\vp_T$ is uniquely determined. Let $\psi$ be any $T$-normalized $\Gamma$-gain function that, like $\vp_T$, is shifting equivalent to $\vp$. If $e$ is in the tree $T_1$ of $T$, then $T_1-e$ consists of two trees connected by a bond $B$ in which $e$ is the only edge of $T$, thus the only edge in $B$ for which $\psi$ may be nonzero. Orienting $e$ and $B$ compatibly, we have
\begin{equation}\label{E:Tnormalized}
\psi(e) = \psi(\vec B) = \vp(\vec B).
\end{equation}
In particular, $\vp_T(e) = \vp(\vec B) = \psi(e)$ for $e$ in $T$.  Thus, $\psi = \vp_T$.
\end{proof}

\begin{proof}[Proof of Theorem \ref{T:ShiftingEquivalent}]
If $\vp$ and $\psi$ are shifting equivalent, then $\vp(\vec B)=\psi(\vec B)$ for all oriented bonds $\vec B$ in $G$ from the definition of shifting. Conversely, assume that $\vp(\vec B)=\psi(\vec B)$ for all oriented bonds $\vec B$ in $G$. Let $T$ be a maximal forest in $G$.
Then $\psi_T=\psi=\vp=\vp_T$ on bonds, which implies  $\psi_T(e) = \vp_T(e)$ for all edges by equation \eqref{E:Tnormalized}.  It follows that $\psi$ is shifting equivalent to $\vp$.
\end{proof}

\subsection{Cobiased graphs from vertex labelings}

Let $G$ be a graph and let $\Gamma$ be an additive group. Consider a vertex labeling $\pi\colon V(G)\to\Gamma$. If $H$ is a subgraph of $G$, we write $\pi(H)=\sum_{v\in V(H)}\pi(v)$. Call $\pi$ a $\Gamma$-\emph{quotient labeling} when for each connected component $H$ of $G$, $\pi(H)=0$. Now if $B=\delta(X)$ is a bond of $G$ and $\vec B$ is an orientation of $B$ towards $X$, then define $\pi(\vec B)=\pi(X)$ and say that $B$ is cobalanced when $\pi(\vec B)=0$. Let $\mc L_\pi$ be the set of cobalanced bonds relative to $\pi$.

Such $\Gamma$-quotient labelings were used by Recski \cite{Recski1,Recski2}, for $\Gamma$ equal to the additive group of a field, to characterize vector representations of single-element extensions and elementary quotients of graphic matroids over fields as well as defining some more general extension and elementary-quotient constructions for graphic matroids.

\begin{proposition} \label{P:QuotientLabeling}
Let $\Gamma$ be an additive group and $\pi$ a $\Gamma$-quotient labeling of a graph $G$. Then $(G,\mc L_\pi)$ is a cobiased graph.
\end{proposition}
\begin{proof}
Consider a tribond containing the three bonds $\delta(X_1)$, $\delta(X_2)$, $\delta(X_3)$ and assume without loss of generality that $\delta(X_1),\delta(X_2)\in\mc L_\pi$. Let $\vec B_1$ and $\vec B_2$ be the oriented bonds obtained from $\delta(X_1)$ and $\delta(X_2)$ by orienting all edges away from $X_1$ and $X_2$. Let $\vec B_3$ be the oriented bond obtained from $\delta(X_3)$ by orienting all of its edges towards $X_3$. Now
\[\pi(\vec B_3)=\pi(X_3)=-\pi(X_1\cup X_2)=-\pi(X_1)-\pi(X_2)=\pi(\vec B_1)+\pi(\vec B_2)=0,\]
which implies that a tribond cannot have exactly two bonds in $\mc L_\pi$, which implies our result.
\end{proof}

It might seem that quotient labelings are different than gains; however, they are actually equivalent constructions. Theorem  \ref{T:GainsFromLabelings} describes how to get gains from a quotient labeling and Theorem \ref{T:GainsToLabelings} describes how to get a quotient labeling from gains.

Given a $\Gamma$-quotient labeling $\pi$ of $G$ and a maximal forest $T$ in $G$, define a $T$-normalized $\Gamma$-gain function $\vp_{\pi,T}$ as follows. For each oriented edge $e$ not in $T$, let $\vp_{\pi,T}(e)=0$. For an edge $e$ in $T$, let $B=\delta(X)$ be the bond $\exterior(\pi_{T\bs e})$ and say that $\vec B$ is the orientation of $B$ directed towards $X$. Orient $e$ towards $X$ as well. Now set $\vp_{\pi,T}(e)=\pi(X)$.

\begin{theorem} \label{T:GainsFromLabelings}
Let $\Gamma$ be an additive group and $\pi$ a $\Gamma$-quotient labeling of graph $G$. If $T$ is a maximal forest of $G$, then for every oriented bond $\vec B$ in $G$, $\vp_{\pi,T}(\vec B)=\pi(\vec B)$.
\end{theorem}

Proposition \ref{P:GainsFromLabelings} is necessary for the proof of Theorem \ref{T:GainsFromLabelings}.

\begin{proposition} \label{P:GainsFromLabelings}
Let $\Gamma$ be an additive group and $\pi$ a $\Gamma$-quotient labeling of graph $G$. If $T$ and $T'$ are maximal forests in $G$, then $\vp_{\pi,T}$ and $\vp_{\pi,T'}$ are shifting equivalent.
\end{proposition}
\begin{proof}
One shift operation can be performed in each connected component of $G$. Thus the result is true if and only if it is true for connected graphs, so we may assume that $G$ is connected. Consider the following well-known operation on spanning trees, which we will call \emph{edge exchange} in this proof. If $T$ is a spanning tree of $G$ and $e\notin T$, then for any edge $f\neq e$ on the unique cycle in $T\cup e$, $(T\bs f)\cup e$ is a spanning tree of $G$. It is well known that if $G$ is a connected graph and $T$ and $T'$ are spanning trees of $G$, then there is a sequence of spanning trees $T_1,\ldots,T_k\subseteq(T\cup T')$ such that $T=T_1$, $T'=T_k$, and $T_{i+1}$ is obtained from $T_i$ by an edge exchange. So to complete the proof it suffices to show that $\vp_{\pi,T_{i+1}}$ is obtained from $\vp_{\pi,T_i}$ by a single shift operation.

Say that $T_{i+1}=(T_i\bs f)\cup e$ and let $g\notin\{e,f\}$ be any other edge in $T_i$. Thus $T_i\bs\{f,g\}$ has exactly three connected components $S_1,S_2,S_3$ as shown in Figure \ref{F:VertexLabelingsToGains}. Orient edges $f$ and $g$ as indicated. There are two cases for the placement of $e$ relative to $f$ and $g$ as shown in the figure.

\begin{figure}[H]
\begin{center}
\includegraphics[page=5,scale=1]{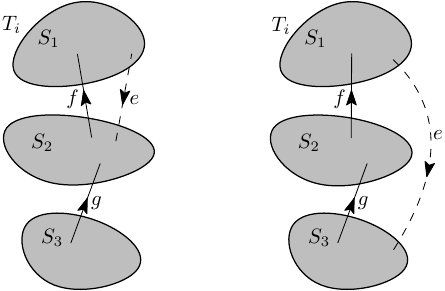}
\end{center}
\caption{Figure for the proof of Proposition \ref{P:GainsFromLabelings}.} \label{F:VertexLabelingsToGains}
\end{figure} Let $C$ be the unique cycle in $T_i\cup e$ oriented in the opposite direction to $e$ and $f$. Note that $g$ is in $C$ in the right configuration of Figure \ref{F:VertexLabelingsToGains} but not in the left configuration. Now let $a=\vp_{{\pi,T_i}}(f)$. We prove that $\vp_{{\pi,T_i}}+\psi_{{C,a}}=\vp_{{\pi,T_{i+1}}}$, which will satisfy our requirement. First, by definition, $(\vp_{{\pi,T_i}}+\psi_{{C,a}})(f)=0=\vp_{{\pi,T_{i+1}}}(f)$ and $(\vp_{{\pi,T_i}}+\psi_{{C,a}})(e)=-a=\vp_{{\pi,T_{i+1}}}(e)$. Second, for the configuration on the left of Figure \ref{F:VertexLabelingsToGains}, \[\vp_{{\pi,T_i}}(g)=(\vp_{{\pi,T_i}}+\psi_{{C,a}})(g)=-\pi(S_3)=\vp_{{\pi,T_{i+1}}}(g).\] Finally, for the configuration on the right, \[(\vp_{{\pi,T_i}}+\psi_{{C,a}})(g)=\vp_{{\pi,T_i}}(g)-a=-\pi(S_1)-\pi(S_3)=-\pi(S_1\cup S_3)=\vp_{{\pi,T_{i+1}}}(g).\] Since $g$ was chosen arbitrarily we have proven that $\vp_{{\pi,T_i}}+\psi_{{C,a}}=\vp_{{\pi,T_{i+1}}}$.
\end{proof}

\begin{proof}[Proof of Theorem \ref{T:GainsFromLabelings}]
Let $B$ be a bond of $G$. If $B$ intersects $T$ in one edge (i.e., $B=\exterior(\pi_{T\bs e})$) then $\pi(\vec B)=\vp_{\pi,T}(\vec B)$ by the definition of $\vp_{\pi,T}$. If $|B\cap T|\geq2$, then let $T'$ be any maximal forest for which $|B\cap T'|=1$. By Proposition \ref{P:GainsFromLabelings} $\vp_{\pi,T'}$ and $\vp_{\pi,T}$ are shifting equivalent, so $\pi(\vec B)=\vp_{\pi,T'}(\vec B)=\vp_{\pi,T}(\vec B)$, as required.
\end{proof}

Conversely, assume that $(G,\vp)$ is a $\Gamma$-gain graph. Assume that $G$ is completely split apart along cut vertices so that each connected component of $G$ is a block. This operation, of course, does not change the join matroids (Theorem \ref{T:CutVertices}). To simplify the discussion, assume that $G$ is loopless and has no isolated vertices. (Loops in $G$ are always loops in the join matroids and isolated vertices have no effect on the matroids.) Therefore, for each vertex $v$, $\delta(v)$ is a bond of $G$. Let $\vec B_v$ be the bond $\delta(v)$ oriented towards $v$. Define $\pi_\vp(v)=\vp(\vec B_v)$.

\begin{theorem} \label{T:GainsToLabelings}
If $(G,\vp)$ is a $\Gamma$-gain graph in which each connected component is a block, then $\pi_\vp$ is a quotient labeling and for each oriented bond $\vec B$ in $G$, $\vp(\vec B)=\mc L_{\pi_\vp}(\vec B)$.
\end{theorem}
\begin{proof}
For each connected component $H$ of $G$, \[\sum_{v\in V(H)}\pi_\vp(v)=\sum_{v\in V(H)}\vp(\vec B_v)=\sum_{e\in\vec E(H)}\vp(e)=0.\] Thus $\pi_\vp$ is a quotient labeling. In a similar fashion, if $\vec B$ is the orientation of bond $\delta(X)$ directed towards $v$, then \[\sum_{v\in X}\pi_\vp(v)=\sum_{v\in X}\vp(\vec B_v)=\sum_{e\in\vec B}\vp(e)=\vp(\vec B),\] as required.
\end{proof}

\begin{example}[Nonseparating bonds]\label{X:Nonseparating}
We generalize the examples of Section \ref{S:TwoPointCobias} to an arbitrary subset $W$ of $V(G)$ by defining $\mc L_{W}$ as the set of bonds $\delta(X)$ that do not separate $W$; i.e., $W \subseteq X$ or $W \subseteq V(G)\bs X$.  It is easy to verify that this set is a linear class, but it is not additive (in the sense of Section \ref{S:TwoPointCobias}) if $|W|>2$ since it is possible for every part of the tripartition of a tribond to contain a vertex of $W$.

This linear class exemplifies $\Gamma$-quotient labeling with the group $\Gamma = \mathbf{Z}_{|W|}$.  The label of a vertex is $0$ is $\pi(v)=0$ if $v \notin W$ and $\pi(v) = 1$ if $v \in W$.  Since $\pi(\delta(X)) \equiv |W \cap X| \mod |W|$ for any $X \subseteq V(G)$, only nonseparating bonds are cobalanced relative to $\pi$.

We interpret Example \ref{X:NonseparatingW} in a similar way.  Using the same group $\mathbb{Z}_2$, assign vertex values as in the previous example.
\end{example}

\subsection{Deletions and contractions for gains} \label{S:DeletionsContractionsGains}

Let $\vp$ be a $\Gamma$-gain function. The set $\mc L_\vp/e$ is defined in Section \ref{S:DeletionsContractions1} as the set of bonds in $\mc L_\vp$ that do not contain $e$ and it is shown that $\mc L_\vp/e$ is a linear class of bonds of $G/e$. Let $\vp/e$ be the $\Gamma$-gain function defined on $G/e$ by restriction of $\vp$ to $E(G/e)=E(G)\bs e$. Proposition \ref{P:GainContraction} is immediate.

\begin{proposition} \label{P:GainContraction}
If $\vp$ is a $\Gamma$-gain function and $e$ is an edge of $G$, then $\mc L_\vp/e=\mc L_{\vp/e}$.
\end{proposition}

The set $\mc L_\vp\bs e$ is defined in Section \ref{S:DeletionsContractions1} as the set of bonds $B$ in $G\bs e$ for which either $B$ or $B\cup e$ is a bond in $\mc L_\vp$. As long as $e$ is not an isthmus of $G$, there is a $\Gamma$-gain function $\psi$ on $G$ that is shifting equivalent to $\vp$ and for which $\psi(e)=0$. (See Proposition \ref{P:Normalizing}.) Define $\psi\bs e$ to be the $\Gamma$-gain function on $G\bs e$ defined by restriction of $\psi$ to $E(G\bs e)=E(G)\bs e$.

\begin{proposition} \label{P:GainDeletion}
If $\vp$ is a $\Gamma$-gain function and $e$ is a non-isthmus edge of $G$, then there is $\psi$ that is shifting equivalent to $\vp$ such that $\psi(e)=0$ and $\mc L_\vp\bs e=\mc L_\psi\bs e=\mc L_{\psi\bs e}$.
\end{proposition}
\begin{proof}
The existence of $\psi$ is implied by Proposition \ref{P:Normalizing}. Now, if $B\in\mc L_{\psi\bs e}$, then $(\psi\bs e)(\vec B)=0$. Since $B$ is a bond of $G\bs e$, there is a bond $B_e\in\{B,B\cup e\}$ of $G$. Since $\psi(e)=0$, we get $\vp(\vec B_e)=\psi(\vec B_e)=0$. This implies that $B_e\in\mc L_\vp$, which implies that $B\in\mc L_\vp\bs e$. Conversely, if $B\in\mc L_\vp\bs e$, there is a bond $B_e\in\{B,B\cup e\}$ of $\mc L_\vp$ for which $\psi(\vec B_e)=\vp(\vec B_e)=0$. This implies that $(\psi\bs e)(\vec B)=0$, which makes $B\in\mc L_{\psi\bs e}$.
\end{proof}

\subsection{Deletions and contractions for quotient labelings} \label{S:DeletionsContractionsQuotients}

Let $\pi$ be a $\Gamma$-quotient labeling of $G$ and let $e$ be a link in $G$ with endpoints $u$ and $v$. Let $w$ be the vertex obtained by the contraction of $e$ in $G$. Define $\pi/e$ to be the labeling on $V(G/e)$ given by $(\pi/e)(x)=\pi(x)$ when $x\in V(G)\cap V(G/e)$ and $(\pi/e)(w)=\pi(u)+\pi(v)$.

\begin{proposition} \label{P:QuotientLabelingContraction}
If $\pi$ is a $\Gamma$-quotient labeling of $G$ and $e$ is a link in $G$, then $\pi/e$ is a $\Gamma$-quotient labeling of $G/e$ and $\mc L_{\pi}/e=\mc L_{\pi/e}$.
\end{proposition}
\begin{proof}
A bond $B$ of $G/e$ is in $\mc L_{\pi}/e$ if and only if $B$ is a bond of $G$ not containing $e$ and $B\in\mc L_\pi$ if and only if both endpoints of $e$ are in $X$ or both endpoints of $e$ are not in $X$ where  $B=\delta(X)$ if and only if $B$ is a bond of $G/e$ in $\mc L_{\pi/e}$.
\end{proof}

For Proposition \ref{P:QuotientLabelingDeletion}, we define the vertex labeling $\pi\bs e$ on $G\bs e$ by $\pi\bs e=\pi$.

\begin{proposition} \label{P:QuotientLabelingDeletion}
If $\pi$ is a $\Gamma$-quotient labeling of $G$ and $e$ is a link in $G$, then $\pi\bs e$ is a $\Gamma$-quotient labeling of $G\bs e$ if and only if $e$ is not an un-cobalanced isthmus of $(G,\mc L_\pi)$. Furthermore, if $e$ is not an un-cobalanced isthmus of $(G,\mc L_\pi)$, then $\mc L_{\pi}\bs e=\mc L_{\pi\bs e}$.
\end{proposition}
\begin{proof}
The first statement is obvious. Now if $e$ is not an un-cobalanced isthmus, consider a bond $B$ in $G\bs e$ and let $B_e\in\{B,B\cup e\}$ be the corresponding bond in $G$. Now $B\in\mc L_{\pi}\bs e$ if and only $B_e\in\mc L_\pi$ if and only if $B\in \mc L_{\pi\bs e}$.
\end{proof}

\subsection{Gains and quotient labelings using fields}

Let $\bb F$ be any field. Denote the additive group of $\bb F$ by $\bb F^+$. Scalar multiplication of $\bb F^+$-gain functions and quotient labelings using a nonzero element of $\bb F$ does not affect the resulting linear class of bonds. The proof of Proposition \ref{P:ScalarMultiplication} is evident. Readers who are familiar with partial fields will note that this operation generalizes immediately to partial fields.

\begin{proposition} \label{P:ScalarMultiplication}
If $(G,\vp)$ is an $\bb F^+$-gain graph, $\pi$ an $\bb F^+$-quotient labeling, and $a$ a nonzero element of $\bb F$, then
\begin{itemize}
\item[(1)] $\mc L_\vp=\mc L_{a\vp}$ and
\item[(2)] $a\pi$ is a quotient labeling of $G$ with $\mc L_{a\pi}=\mc L_\pi$.
\end{itemize}
\end{proposition}

%%%%%%%%%%%%%%%%%%%%%%%%%%%%%%%%%%%%%%%%%%%%%%%%

\end{document}